\documentclass[reqno]{amsart}
\usepackage[latin1]{inputenc}
\usepackage{amssymb}
\usepackage{amsfonts}
\usepackage{paralist}
\usepackage{enumitem}
\usepackage{indentfirst}
\usepackage{amsmath}
\usepackage[integrals]{wasysym}
\usepackage{amsthm}
\usepackage{mathrsfs}
\usepackage{standalone}
\usepackage{color}
\usepackage{adforn}
\usepackage{comment}
\usepackage[bookmarks]{hyperref}
\usepackage{cleveref}
\usepackage{cancel}
\usepackage{faktor}
\usepackage{mdframed}
\usepackage[normalem]{ulem}
\usepackage{todonotes}

\newcommand{\A}{\mathbb{A}}

\newcommand{\RR}{\mathbb{R}^2}
\newcommand{\R}{\mathbb{R}}

\newcommand{\TT}{\mathbb{T}^2}
\newcommand{\TTT}{\mathbb{T}^3}
\newcommand{\ZZ}{\mathbb{Z}^2}
\newcommand{\Z}{\mathbb{Z}}

\newcommand{\cC}{\mathcal{C}}
\newcommand{\cF}{\mathcal{F}}

\newcommand{\tf}{\tilde{f}}

\newcommand{\wt}{\widetilde}

\DeclareMathOperator{\area}{area}
\DeclareMathOperator{\length}{length}

\newtheorem{defi}{Definition}[section]

\newtheorem*{thm*}{Theorem}
\newtheorem{lemma}[defi]{Lemma}
\newtheorem{rmk}[defi]{Remark}
\newtheorem{theorem}{Theorem}
\newtheorem{claim}{Claim}
\newtheorem{cor-thm}[defi]{Corollary}
\newtheorem*{sc-lemma}{Semicontinuity Lemma}

\newtheorem*{thm1'}{Theorem 1'}

\begin{document}

\title[Absolutely PH surface endomorphisms are dynamically coherent]{Absolutely partially hyperbolic surface endomorphisms are dynamically
coherent}
\author[M. Andersson]{M. Andersson} 
\address{Instituto de Matem\'{a}tica Aplicada. Universidade Federal Fluminense. Rua Professor Marcos Waldemar de Freitas Reis, S/N. 24210-201 Niter\'{o}i, Brazil.}

\author[W. Ranter]{W. Ranter}
\address{Wagner Ranter: Instituto de Matem\'{a}tica. Universidade Federal de Alagoas, Campus A.S. Simoes S/N, 57072-090. Macei\'o, Alagoas, Brazil.}

\email{wagnerranter@im.ufal.br}

\thanks{$^{**}$W.R. was supported by CNPq/MCTI/FNDCT project: 420353/2025-9, and by CNPq, project: 446515/2024-8, Brazil.}
\urladdr{im.ufal.br/professor/wranter}


\date{\today}

\begin{abstract}
	We show that if an endomorphism $f:\TT \to \TT$ is absolutely partially
	hyperbolic, then it has a center foliation. Moreover, the center
	foliation is leaf conjugate to that of its linearization.
\end{abstract}

\maketitle


\section{Introduction and statement of results}

Throughout the last few decades, a central theme in dynamical systems has been
to find suitable generalizations of the well-understood world of
uniformly hyperbolic systems. One major field of study that was born from this
approach is that of \emph{partial hyperbolicity}.

In its original form, a partially hyperbolic system is one which, on the
infinitesimal level, can be decomposed into three (possibly multidimensional)
\emph{directions} ---  an \emph{unstable} one in which there is uniform
expansion, a \emph{stable} one in which there is uniform contraction, and an
intermediate \emph{center} direction which never expands as much as the unstable
one, nor contracts as much as the stable. When making this notion precise, a
seemingly innocent difference in the definition considered by different authors
crept into the literature.

For some authors (starting with \cite{BP}), the gap between the greatest expansion in the center direction
and the weakest expansion in the unstable direction is \emph{absolute}, meaning
that one can compare the expansion rate at different points in the phase space
and still see the gap. For
others, (starting with \cite{HPS}), such a gap is only \emph{pointwise}, meaning
that, at some points, the
expansion in the center direction may be larger than the expansion in the
unstable direction \emph{at another point}. Similarly for the strongest
contraction in the center direction vs. the weakest contraction in the stable
one. The first type of partial hyperbolicity later came to be known as
\emph{absolute} partial hyperbolicity, while the latter form came to be known as
\emph{pointwise} partial hyperbolicity.

Although pointwise partial hyperbolicity is the more general of the two, for a
long time it was assumed that the two notions were essentially equivalent, in
the sense that anything (of interest) that could be proved in the absolute
setting would also be true using the  pointwise definition. This, however,
changed in the 2010's. First it was proved 
in \cite{BBI} that absolutely partially hyperbolic diffeomorphisms on $\TTT$ are
\textit{dynamically coherent}, meaning that combining the center direction with either of
the uniform ones yields an integrable plane distribution. (In particular, the
central direction integrates to a foliation of curves.) This result was further
strengthened 
in \cite{H}, who showed that such systems are \emph{leaf conjugate} to their linearizations.
This means that there is a homeomorphism sending the leaves of all aforementioned foliations to the
corresponding invariant foliations of the unique linear automorphism to which the
system is homotopic. On the other hand --- and here comes the surprise --- it
was shown a few years later 
in \cite{RHRHU} that there are \emph{dynamically incoherent} pointwise 
partially hyperbolic diffeomorphism on $\TTT$.

The last few years has seen a burgeoning interest in \emph{non-invertible} dynamical systems,
often referred to as \emph{endomorphisms}. More specifically, in this work, by
\emph{endomorphism} we mean a $C^1$-local diffeomorphisms on $\TT$ of topological
degree strictly larger than one. In this context, the notion of partial hyperbolicity must be adapted.

A commonly used definition is the existence of a continuous invariant cone field
in the tangent bundle in which vectors are uniformly expanded by the derivative.
More precisely,
an endomorphism $f:\mathbb{T}^2 \to \mathbb{T}^2$ is said to be 
\textit{partially hyperbolic} if it admits a $Df$-invariant
\textit{unstable cone field $\mathcal{C}^u$}. That is, there exist a constant
$\lambda>1$ and a cone family $\mathcal{C}^u$, which consists of closed cone
$\mathcal{C}^u_p$ in $T_p\mathbb{T}^2$, such that 
\begin{align}\label{domination}
	Df(v) \in int(\mathcal{C}^u_{f(p)}) \quad \text{and}
	\quad	\|Df(v)\|\geq \lambda\|v\|,
\end{align}
for every $p \in \mathbb{T}^2$ and every nonzero vector $v \in \mathcal{C}^u_p$.
The existence of an unstable cone field implies the existence of a
$Df$-invariant \emph{center bundle} defined by 
\begin{align*}
	E^c(p)=\{v \in T_p\mathbb{T}^2| Df^n(v) \notin \mathcal{C}^u_{f^n(p)},
	\,\forall n\geq 0\}.
\end{align*}
Moreover, the unstable cone field $\mathcal{C}^u$ \textit{dominates} the
center bundle $E^c$. That is, there is $0<\delta<1$ such that for every $p \in
\mathbb{T}^2$, every unit vectors $v \in \cC_p^u$ and $w \in E_p^c$ we have
\begin{equation} \label{eq:domination}
	\| Df(w) \| \leq \delta \|Df(v)\|.
\end{equation}

Even here one may distinguish between an absolute and a pointwise version of the
definition. \emph{Pointwise} partial hyperbolicity means that
\eqref{eq:domination} holds whenever $v$ and $w$ are
vectors over the same base point. \emph{Absolute } partial hyperbolicity, on
the other hand,
means that \eqref{eq:domination} holds for all unit vectors $v \in \cC_p^u$ and
$w \in E_q^c$ where $p$ and $q$ can be any points in $\TT$.

Our aim in this short note is to complete the analogy with the case of partially hyperbolic diffeomorphisms on $\TTT$ by proving that:

\begin{theorem}\label{main}
	Every absolutely partially hyperbolic endomorphism on $\TT$ is
	dynamically coherent and leaf conjugate to its linearization.
\end{theorem}

Throughout this work, we say that $f$ is \textit{dynamically coherent}
when there is an invariant foliation whose leaves are tangent to the center
bundle. As usual, by foliation we mean a $C^0$ foliation with $C^1$ leaves.
This foliation is denoted by $\mathcal{W}^c_f$. We say that two partially hyperbolic
endomorphism $f$ and $g$ are \textit{leaf conjugate} if both are dynamically coherent
and there is a homoemorphism $\phi:\mathbb{T}^2 \to \mathbb{T}^2$ such that for each
leaf $\mathcal{L} \in \mathcal{W}^c_f$ one has
$$\phi(\mathcal{L}) \in \mathcal{W}^c_g \ \ \text{and} \ \ \phi(f(\mathcal{L}))=g(\phi(\mathcal{L})).$$

Recall that  every endomorphism $f:\mathbb{T}^2\to \mathbb{T}^2$ is homotopic to a
unique linear endomorphism $A:\mathbb{T}^2 \to \mathbb{T}^2$, induced by the
action of $f$ on the first homology group of $\mathbb{T}^2$. We refer to $A$ as
the \textit{linear part} or \textit{linearization} of $f$, and also denote by
$A$ its lift to the universal cover, which is represented by a non-singular
square matrix with integer entries.

We point out that, since $f$ is not invertible, the existence of an unstable
cone field does not necessarily guarantee the existence of an invariant unstable vector
bundle. When such bundle exists, we say that $f$ is a \textit{specially partially
hyperbolic endomorphism}. In \cite{HSW}, the authors establish dynamical coherence
for endomorphisms that are simultaneously absolutely and specially partially hyperbolic.
Moreover, by adapting the examples constructed in \cite{RHRHU}, they also provide
examples of dynamically incoherent partially hyperbolic endomorphisms.
All of their examples are pointwise partially hyperbolic.

Recently, a rather comprehensive classification of partially hyperbolic
endomorphisms on $\TT$ was made by Hall and Hammerlindl in a series of papers
in \cite{HH1, HH2, HH3}. In the first one, they proved that every partially
hyperbolic endomorphism with hyperbolic linearization is dynamically coherent
and leaf conjugate to its linearization. In the second, they provide a
sufficient condition for a partially hyperbolic endomorphism to be dynamically
coherent and leaf conjugate to its linearization (see Theorem~\ref{ThmHH2}
below).
In the third paper, the authors construct some examples of
partially hyperbolic endomorphisms homotopic to homothetically expanding maps
that are not dynamically coherent. 


\section{Some known results}


The proof of Theorem~\ref{main} uses the following elegant sufficient condition for dynamical coherence provided by Hall and Hammerlindl, namely:

\begin{theorem}[\cite{HH2}] \label{ThmHH2}
	Let $f:\TT \to \TT$ be partially hyperbolic. If $f$ does not admit a periodic
	center annulus, then $f$ is dynamically coherent and leaf conjugate to $A$.
\end{theorem}

Here, a \emph{periodic center annulus} is 
an open subset $\mathbb{A}\subseteq \mathbb{T}^2$, homeomorphic to the
annulus $S^1 \times (0,1)$, satisfying $f^n(\mathbb{A})=\mathbb{A}$ for
some $n\geq 1$, whose boundary consists of either one or two disjoint $C^1$
circles tangent to the center direction. 
In addition, a periodic center annulus
is minimal in the sense that there is no annulus with the same property
contained in it. 
Note that such curves are necessarily circles and, by the
Poincar\'{e}-Benedixon Theorem, they are not homotopic to a point. In other words, they represent a non-trivial element of $\pi_1(\TT)$.

\medskip

The strategy of the proof of Theorem~\ref{main} is to show that an absolutely
partially hyperbolic endomorphism on $\TT$ cannot admit a center annulus. In order to do that, we make use of the following "length vs. volume" lemma. 

\begin{lemma}[\cite{HH1}] \label{le:length_vs_volume}
	There is a $K>0$ such that if $J \subset \RR$ is a $C^1$ curve tangent
	to (the lift of the) unstable cone field, then 
	\[\area(U_{1} (J))> K \length(J),\]
	where $U_1(J)$ is the set of points in $\RR$ whose distance to $J$ are less than $1$.
	
\end{lemma}

\begin{rmk}
	Although it is not stated explicitly in \cite{HH2}, the proof of
	Theorem~\ref{ThmHH2} actually reveals that $E^c$ is \emph{uniquely} integrable,
	meaning that every curve tangent to $E^c$ is necessarily contained in a
	leaf of $\cF^c$. Since our proof of Theorem~\ref{main} uses
	Theorem~\ref{ThmHH2}, the conclusion of Theorem~\ref{main} can be strengthened likewise.
\end{rmk}

\section{Proof of Theorem~\ref{main}}

We fix $f$ as in Theorem~\ref{main} and let $A$ be its linear part. Before starting the proof,  some preliminaries are needed.

\subsection*{Center annuli as extensions over circle maps}

As established in Theorem~\ref{ThmHH2} (\cite{HH2}), a partially hyperbolic
endomorphism may fail to be dynamically coherent and leaf-conjugate to its
linearization when it exhibits a periodic center annulus. 
So, from now on, let us suppose that $f$ is not dynamically coherent and let $\A$
be a periodic center annulus. Upon replacing $f$ by an iterate if necessary, 
it is enough to consider the case in which $\A$ and each of the connected
components of $\partial \A$ are
fixed by $f$. 

Let $\Gamma$ be one of the connected component of $\partial\A$ and let $\gamma: S^1 \to
\TT$ be an injective parametrization of $\Gamma$. Using the canonical identification of
$\pi_1(\TT)$ with $\ZZ$, the homotopy class $[\gamma]$ of $\gamma$ can be seen
as an element of $\ZZ$.  The fact that  $\Gamma$ is fixed by $f$ then translates
into saying that $[\gamma]$ is
necessarily an eigenvector of $A$. Note that this implies that 
any partially hyperbolic endomorphism whose linearization has irrational
eigenvalues cannot admit a periodic center annulus. So
from now on we assume that the  eigenvalues of $A$ are integer numbers.  
Up to a linear change of coordinates (see e.g. \cite{AR1}) we may assume that $[\gamma] = e_1$, so that 
\begin{align}\label{matrix}
	A = \left(
	\begin{matrix}
		\ell & k \\
		0    & m
	\end{matrix}
	\right)
\end{align}
for some $\ell,k,m \in \Z$.

Observe that $\ell$ is the eigenvalue associated to the eigenvector $[\gamma]$.
We now consider a self-cover  $g:S^1 \to S^1$ by setting $g(s) =
\gamma^{-1}\circ f\circ \gamma (s)$. The map $g$ is homotopic to the linear map
$S^1 \to S^1$, $s \mapsto \ell s \mod \Z$. 
In particular,  $|\ell|$
is the topological degree (i.e. the number of sheets) of $f\vert_\Gamma: \Gamma
\to \Gamma$.

Note that, although $f(\A) = \A$, we do not necessarily have $f^{-1}(\A) = \A$. However, we can state the following.

\begin{lemma}\label{annulus map}
	The restriction of $f$ to $\A$ is a self-cover whose degree (number of
	sheets) is equal to $|\ell|$.
\end{lemma}

\begin{proof}
Denote the closure of $\A$ by $\overline{\A}$.	Clearly, $f \vert_\A$ and $f \vert_{\overline{\A}}$ have the same
	degree. Now, since $f(\A) = \A$, it follows that every pre-image of a
	point in $\partial \A$ under $f \vert_{\overline{\A}}$ must belong to $\partial \A$. Hence $f \vert_{\partial \A}$ and $f\vert_\A$ have the same degree.	
\end{proof}

Let $\pi:\mathbb{R}^2 \to
\mathbb{T}^2$ be the natural projection and let $\wt{\A}$ denote the closure
of one
connected component of $\pi^{-1}(\A)$.  Now consider a lift $\wt{f}: \RR \to \RR$ of
$f$ which preserves $\wt{\A}$.  
The restriction of $\wt{f}$ to $\wt{\A}$  satisfies
$$
\widetilde{f}\vert_{\wt{\A}}(x+1,y)=\widetilde{f}\vert_{\wt{\A}}(x,y)+(\ell,0),
\quad \forall (x,y) \in \wt{\A}.
$$

Since $\A$ is homeomorphic to the annulus $S^1\times (0,1)$ and its boundary
are circles, we have that $\wt{\A}$ is homeomorphic to $\mathbb{R}\times [0,1]$
and $\wt{f}\vert_{\wt{\A}}$ is topologically conjugate to a map $F:\R \times [0,1] \to \R \times [0,1]$ satisfying

\begin{align}\label{eq.F}
F(x+1,y)=F(x,y)+(\ell,0), \forall (x,y) \in \mathbb{R}\times [0,1].
\end{align}

Now we will see that, wnenever we have a center annulus fixed by $f$, the degree the
restriction of $f$ to the annulus must be greater than one. 
In particular, if $A$ has an eigenvalue of modulus one, 
the boundary of the center annulus represents (in homology) an eigenvector
associated to the larger eigenvalue.





\begin{lemma} \label{le:ell_greater_than_one}
The integer $\ell$ has modulus greater than one.
\end{lemma}

\begin{proof}
	Let $J$ be an unstable curve of finite length in $\wt{A}$.  Suppose, for
	the purpose of contradiction, that $|\ell| = 1$.  Then, since $f \vert
	\A : \A \to \A$ is a diffeomorphism, the diameter of $\tf^n(J)$ grows at
	most linearly (see e.g. \cite{Botelho}). Therefore, $U_1(\tf^n(J))$ is
	contained in a ball whose radius grows at most linearly, and so does
	the area of $U_1(\wt{f}^n(J))$. But the length of $J$ grows
	exponentially, 
	contradicting the length vs. volume lemma (Lemma~\ref{le:length_vs_volume}). 
	Therefore, $\ell$ must be greater than one.
\end{proof}

The next result states that the restriction of $f$ to the closure of $\A$ is semi-conjugate to the expanding map $x \mapsto \ell x \mod \Z$ on $S^1$. In fact, this follows from a simple modification of Frank's proof of semi-conjugacy in the class of Anosov diffeomorphisms and is probably somewhat folkloric. However, we found no reference to this fact, except for  \cite[Corollary~5]{IPRX}), where it follows as a corollary of a (much more involved) result about self-covers of the open annulus (which in general do not extend to self-covers of the closed annulus).  For completeness we include a proof using a variant of
Franks' construction.

\begin{lemma}\label{le:semi}
	Let $I = [0,1]$ and let $F:\R \times I \to \R \times I$ be a continuous map satisfying \eqref{eq.F} and some integer
	$\ell$ of modulus greater than one. Then there exists a continuous
	surjection $H:\RR \times I \to \R$ satisfying $H(x+1,y) = H(x,y)+1$
	for every $(x,y) \in \R \times I$, such that $H\circ F = \ell H$.
\end{lemma}

\begin{proof}
	Let 
	\[C_{\Z}^0 = \{\varphi \in C^0(\R \times I, \RR): \varphi(x+1,y)
		=\varphi(x,y) \ \forall (x,y) \in \R \times I\}
	\]
	and consider the operator $\cF: C_{\Z}^0 \to C_{\Z}^0$ defined by
	\[\cF \phi (x,y) = \frac{1}{\ell}\left( P \circ F(x,y)-\ell x + \varphi \circ
	F(x,y) \right),\]
	where $P: \R \times I \to \R$ is the projection to the first coordinate.
	It is straightforward to check that $\cF$ is well defined, i.e. that
	$\cF \varphi \in C_{\Z}^0$ whenever $\varphi \in C_{\Z}^0$. Moreover, if
	$\psi$ is a fixed point of $\cF$, then 
	\[\ell(x+u) = (P+u)\circ F.\]
	In other words, $H = P+u$ satisfies the conclusion of the lemma. 

	Note that 
	\[\cF \varphi - \cF \phi = \frac{1}{\ell}(\varphi- \phi) \circ F,\]
	so that, in particular,
	\[ \|\cF \varphi - \cF \phi \|_{C^0} \leq \frac{1}{|\ell|} \| \varphi -
	\phi \|_{C^0}.\]
	In other words, $\cF$ is a contraction on $C_{\Z}^0$ in the usual sup
	norm, so it has a (unique) fixed point. This completes the proof.
\end{proof}

\subsection*{Proof of Theorem~\ref{main}}
We are now able to prove the main result of this note.
 So far we have dealt with generalities. In particular, nothing said so far depends on \emph{absolute} partial hyperbolicity. 
Now we explore how the existence of a periodic center annulus in an absolutely partially hyperbolic setting would imply that the expansion rate in the unstable direction must be greater than the
degree of the restriction of $f$ to the center annulus. Recall our current
setting: We are considering an absolutely partially hyperbolic endomorphism
$f:\TT \to \TT$. We further suppose (for the purpose of arriving at a
contradiction) that $f$ \emph{fixes} a center annulus $\A$ and, moreover, that $f$ also
fixes both connected components of $\partial \A$ (in case there are more than
one). We fix one such connected component and call it $\Gamma$ and denote by
$\ell$ the (signed) degree of $f \vert \Gamma$.

\begin{claim}
There exists $c>0$ and 	$\lambda> |\ell|$ such that $\| Df^n v \| \geq c \lambda^n \| v\|$, for every non-zero $v \in \cC_x^u$ and every $n \geq 1$.
\end{claim}

\begin{proof}[Proof of Claim~1]
Sine the restriction of $f$ to $\Gamma$ is a local diffeomorphism of the circle
to itself with topological degree $|\ell|$, the integral of the Jacobian of $f \vert \Gamma$ is equal to $|\ell|$ times the arc-length of $\Gamma$. 
By the Mean Value Theorem there must be some point $p$ on $\Gamma$ where the 	
Jacobian is at least $|\ell|$. Since $\Gamma$ is tangent to the center direction, every tangent vector to $\Gamma$ lies in $E^c$. In particular, the Jacobian of $f \vert_\Gamma$ at $p$ is simply $\|Df(u)\|$, where $u$ is a unit vector tangent to $\Gamma$ (hence $u \in E^c(p)$).  Therefore, the number $\lambda=\min_{p \in \mathbb{T}^2}\{\|Df(v)\|/\|v\|: v \in \mathcal{C}^u_p\backslash\{0\}\}$, which by \eqref{eq:domination} must be larger than $|\ell|$. 
\end{proof}

Next, let $\widetilde{\mathbb{A}}$ be a connected component of the pre-image of
the universal covering of $\A$, and fix a lift $\widetilde{f}: \RR \to \RR$ of
$f$ that preserves $\widetilde{\A}$. We now consider some unstable line segment
$J$ contained in $\widetilde{\mathbb{A}}$. By the claim above together with the
length vs. volume lemma (Lemma~\ref{le:length_vs_volume}), we have that
\begin{align}
	\area(U_{1} (\widetilde{f}^n(J)))> K \length(\widetilde{f}^n(J))
	\geq Kc\lambda^n \length(J).
\end{align}

Let $\widetilde{h}$ be the map given in Lemma~\ref{le:semi} such that $\wt{h}\circ \wt{f}\vert_{\wt{A}}=\ell\wt{h}$. Then, we have that $$\widetilde{h}(\widetilde{f}^n(J)) = \ell^n (\widetilde{h}(J))$$
is a horizontal line segment of length $|\ell|^n$ times the length of $\widetilde{h}(J)$. Therefore, $U_1(\widetilde{f}^n(J))$ must be contained on a rectangle of width $|\ell|^n\cdot \length(\widetilde{h}(J))+2\|\widetilde{h}-P\|+2$ and height $2\|\widetilde{h}-P\|+2$. But Lemma~\ref{le:length_vs_volume} says that the area of $U_1(\widetilde{f}^n(J))$ must be at least $C\lambda^n$ for some $C$. That is a contradiction for large $n$. We conclude that $f$ cannot have a periodic center annulus. Hence, by Theorem~\ref{ThmHH2}, $f$ is dynamically coherent and	leaf conjugate to its linearization. \qed


\bibliographystyle{alpha}
\bibliography{biblio}

\begin{thebibliography}{RHRHU16}

\bibitem[AR25]{AR1}
M.~Andersson and W.~Ranter.
\newblock Partially hyperbolic endomorphisms with expanding linear part.
\newblock {\em Ergodic Theory Dyn. Syst.}, 45(2):321--336, 2025.

\bibitem[BBI09]{BBI}
M.~Brin, D.~Burago, and S.~Ivanov.
\newblock Dynamical coherence of partially hyperbolic diffeomorphisms of the
  3-torus.
\newblock {\em J. Mod. Dyn.}, 3(1):1--11, 2009.

\bibitem[Bot88]{Botelho}
F.~Botelho.
\newblock Rotation sets of maps of the annulus.
\newblock {\em Pac. J. Math.}, 133(2):251--266, 1988.

\bibitem[BP74]{BP}
M.~I. Brin and Ya.~B. Pesin.
\newblock Partially hyperbolic dynamical systems.
\newblock {\em Izv. Akad. Nauk SSSR, Ser. Mat.}, 38:170--212, 1974.

\bibitem[Ham13]{H}
Andy Hammerlindl.
\newblock Leaf conjugacies on the torus.
\newblock {\em Ergodic Theory Dyn. Syst.}, 33(3):896--933, 2013.

\bibitem[HH21]{HH1}
L.~Hall and A.~Hammerlindl.
\newblock Partially hyperbolic surface endomorphisms.
\newblock {\em Ergodic Theory Dyn. Syst.}, 41(1):272--282, 2021.

\bibitem[HH22a]{HH2}
L.~Hall and A.~Hammerlindl.
\newblock Classification of partially hyperbolic surface endomorphisms.
\newblock {\em Geometriae Dedicata}, 216(29):1572--9168, March 2022.

\bibitem[HH22b]{HH3}
L.~Hall and A.~Hammerlindl.
\newblock Dynamically incoherent surface endomorphisms.
\newblock {\em Journal of Dynamics and Differential Equations}, March 2022.

\bibitem[HPS77]{HPS}
M.~W. Hirsch, C.~C. Pugh, and M.~Shub.
\newblock {\em Invariant manifolds}, volume 583 of {\em Lect. Notes Math.}
\newblock Springer, Cham, 1977.

\bibitem[HSW19]{HSW}
B.~He, Y.~Shi, and X.~Wang.
\newblock Dynamical coherence of specially absolutely partially hyperbolic
  endomorphisms on {$\mathbb{T}^2$}.
\newblock {\em Nonlinearity}, 32(5):1695--1704, apr 2019.

\bibitem[IPRX16]{IPRX}
J.~Iglesias, A.~Portela, A.~Rovella, and J.~Xavier.
\newblock Dynamics of covering maps of the annulus i: semiconjugacies.
\newblock {\em Mathematische Zeitschrift}, 284(1):209--229, 2016.

\bibitem[RHRHU16]{RHRHU}
F.~Rodriguez~Hertz, M.~A. Rodriguez~Hertz, and R.~Ures.
\newblock A non-dynamically coherent example on $\mathbb{T}^3$.
\newblock {\em Ann. Inst. Henri Poincar{\'e}, Anal. Non Lin{\'e}aire},
  33(4):1023--1032, 2016.

\end{thebibliography}

\end{document}